\documentclass{amsart}

\usepackage{amssymb,amsmath,latexsym,amsthm}
\usepackage[english]{babel}
\usepackage{enumerate}
\usepackage{color}

\newtheorem{proposition}{Proposition}[section]
\newtheorem{theorem}[proposition]{Theorem}
\newtheorem{remark}[proposition]{Remark}
\newtheorem{lemma}[proposition]{Lemma}
\newtheorem{corollary}[proposition]{Corollary}

\newtheorem{conjecture}[proposition]{Conjecture}

\newcommand{\nth}[1]{$#1 {\rm - th }$}

\newcommand{\rad}{{\rm rad } \ }

\newcommand{\rk}{{\rm rank } \ }

\newcommand{\Z}{\mathbb{Z}}
\newcommand{\ZM}[1]{\Z /( #1 \cdot \Z)}
\newcommand{\ZMs}[1]{(\Z / #1 \cdot \Z)^*}

\newcommand{\rg}[1]{\mbox{\bf #1}}
\newcommand{\eu}[1]{\mathfrak{#1}}
\newcommand{\id}[1]{\mathcal{#1}}
\newcommand{\Gal}{\mbox{Gal}}

\newcommand{\GL}{\mbox{GL}}

\newcommand{\rf}[1]{(\ref{#1})}

\newcommand{\Norm}{\mbox{\bf N}}
 
\newcommand{\lchooses}[2]{\left( \frac{#1}{#2 } \right)} 
 
\newcommand{\F}{\mathbb{F}}

\newcommand{\K}{\mathbb{K}}

\newcommand{\Q}{\mathbb{Q}}
\newcommand{\R}{\mathbb{R}}
\newcommand{\C}{\mathbb{C}}
\newcommand{\N}{\mathbb{N}}

\def\ra{\rightarrow}

\begin{document}

\title {On the equation $X^n - 1 = B Z^n$} 
\author{B. BARTOLOM\'{E} \and P. MIH\u{A}ILESCU}


{\obeylines \small
\vspace*{0.2cm}
\hspace*{4cm}When it all comes down to dust
\hspace*{4cm}I will kill you if I must,
\hspace*{4cm}I will help you if I can.
\hspace*{4cm}And mercy on our uniform,
\hspace*{4cm}man of peace or man of war,
\hspace*{4cm}the peacock spreads his fan. 
\vspace*{0.8cm}
\hspace*{5.2cm} To Yuri Bilu, colleague and 
\hspace*{6.5cm} friend for all seasons.\footnote{$ \ $ ''Story of Isaac'',  by Leonard Cohen.}
\smallskip}

\vspace*{0.5cm}
\date{Version 1.13 \today}

\begin{abstract}
  We consider the Diophantine equation $X^n - 1 = B Z^n$, where $B \in
  \Z$ is understood as a parameter. We prove that if the equation has
  a solution, then either the Euler totient of the radical,
  $\varphi(\rad(B))$, has a common divisor with the exponent $n$, or
  the exponent is a prime and the solution stems from a solution to
  the diagonal case of the Nagell--Ljunggren equation: $\frac{X^n -
    1}{X- 1} = n^e Y^n, \ e \in \{0, 1\}$. This allows us to apply
  recent results on this equation to the binary Thue equation in
  question. In particular, we can then display parametrized families
  for which the Thue equation has no solution. The first such family
  was proved by Bennett in his seminal paper on binary Thue equations
  \cite{Be}.
\end{abstract}

\maketitle


\section{Introduction}
Let $B, n \in \N_{>1}$ be such that 
\begin{eqnarray}
  \label{nosplit}
 \varphi^{*}(B) := \varphi(\rad(B)) \quad \hbox{ and } \quad (n,
 \varphi^*(B)) = 1.
\end{eqnarray}
Here $\rad(B)$ is the radical of $B$ and the condition implies that
$B$ has no prime factors $t \equiv 1 \bmod n$. In particular, none of
its prime factors splits completely in the \nth{n} cyclotomic field.

More generally, for a fixed $B \in \N_{> 1}$ we let 
\begin{eqnarray}
\label{expos}
\id{N}(B) = \{ n \in \N_{ > 1} \ | \ \exists \ k > 0 \hbox{ such that } n | {\varphi^*(B)}^k \}. 
\end{eqnarray}

If $p$ is an odd prime, we shall denote by CF the combined condition
requiring that
\begin{enumerate}[I]
\item \label{i} The Vandiver Conjecture holds for $p$, so the class
  number $h_p^+$ of the maximal real subfield of the cyclotomic field
  $\Q[\zeta_p]$ is not divisible by $p$.
\item \label{ii} The index of irregularity of $p$ is small, namely $i_r(p)
  < \sqrt{p} -1$, so there are $i_r(p)$ odd integers $k < p$ such that
  the Bernoulli number $B_k \equiv 0 \bmod p$.
\end{enumerate}
The second condition was discovered by Eichler, as a sufficient
condition for FLT1 to be true. It is known from recent computations of
Buhler and Harvey \cite{BH} that the condition CF is satisfied by
primes up to $163 \cdot 10^6$.

We consider the binary Thue equation
\begin{eqnarray}
\label{bin}
X^n - 1 & = & B \cdot Z^n,
\end{eqnarray}
where solutions with $Z \in \{ -1, 0, 1\}$ are considered to be
trivial. The assertion that equation \rf{bin} has finitely many solutions 
other than the trivial ones is a special case of the general Pillai conjecture
(Conjecture 13.17 of \cite{Bilu}). This equation is encountered as a
particular case of binary Thue equations of the type
\begin{eqnarray}
\label{mike} 
 a X^n - b Y^n = c, 
\end{eqnarray}
see \cite{BGMP}. In a seminal paper \cite{Be}, Michael Bennett proves
that in the case of $c = \pm 1$ there is at most one solution for
fixed $(a, b; n)$ and deduces that the parametric family $(a+1, a; n)$
has the only solution $(1,1)$ for all $n$.  The equation \rf{bin}
inserts naturally in the family of equations \rf{mike}, with $a = c =
\pm 1$.

A conjecture related directly to \rf{bin} states that 
\begin{conjecture}
\label{bt}
Under \rf{nosplit}, Equation \rf{bin} has no other non-trivial solution than $(X, Y; B, n) =
(18, 7; 17, 3)$.
\end{conjecture}
Current results on \rf{bin} are restricted to values of $B$ which are
built up from small primes $p \leq 13$ \cite{G}. If expecting that the equation
has no solutions, -- possibly with the exception of some isolated
examples -- it is natural to consider the case when the exponent $n$
is a prime. Of course, the existence of solutions $(X, Z)$ for
composite $n$ imply the existence of some solutions with $n$ prime, by
raising $X, Z$ to a power. 

The main contribution of this paper will be to relate \rf{bin} 
in the case when $n$ is a prime and \rf{nosplit} holds, to the
diagonal Nagell -- Ljunggren equation, 
\begin{eqnarray}
\label{dn}
\frac{X^n-1}{X-1} = n^e Y^n, \quad e = \begin{cases} 0 & \hbox{ if } X
\not \equiv 1 \bmod n, \\
1 & \hbox{ otherwise.} \end{cases}
\end{eqnarray}
This way, we can apply results from \cite{Mi1} and prove the following:

\begin{theorem}
\label{tbin}
Let $n$ be a prime and $B > 1$ an integer with $(\varphi^{*}(B), n ) =
1$. Suppose that the equation \rf{bin} has a non trivial integer
solution different from $n = 3$ and $(X, Z; B) = (18, 7; 17)$. Let $X
\equiv u \bmod n , 0 \leq u < n$ and $e = 1$ if $u = 1$ and $e = 0$
otherwise. Then:

\begin{enumerate}[A]
\item $n > 163 \cdot 10^6$. \label{A}
\item $X - 1 = \pm B/n^e$ and $B < n^n$. \label{B}
\item If $u \not \in \{ -1, 0, 1\}$, then condition CF \rf{ii} fails
  for $n$ and \label{C}
\begin{eqnarray*}
\begin{array}{r c r r l}  
2^{n-1} & \equiv & 3^{n-1} \equiv 1 & \bmod \ n^2, & \hbox{ and } \\
r^{n-1} & \equiv & 1 & \bmod \ n^2  & \hbox{ for all $ r | X(X^2-1)$}. 
\end{array}
\end{eqnarray*}

If $u \in \{ -1, 0, 1\}$, then Condition CF \rf{i} fails for $n$.
\end{enumerate}
\end{theorem}

The particular solution $n = 3$ and $(X, Z; B) = (18, 7; 17)$ is
reminiscent of a solution of the diagonal Nagell solution; it is
commonly accepted that the existence of non trivial solutions tends to
render Diophantine equations more difficult to solve. Based on Theorem
\ref{tbin} we prove nevertheless the following

\begin{theorem}
\label{genexpo}
If equation \rf{bin} has a solution for a fixed $B$ verifying the
conditions \rf{nosplit}, then either $n \in \id{N}(B)$ or there is a
prime $p$ coprime to $\varphi^*(B)$ and a $m \in \id{N}(B)$ such that
$n = p \cdot m$. Moreover $X^m, Y^m$ are a solution of \rf{bin} for
the prime exponent $p$ and thus verify the conditions in Theorem
\ref{tbin}.
\end{theorem}
This is a strong improvement of the currently known results.

\begin{remark}
\label{remark}
  Theorem \ref{tbin} uses criteria from the diagonal case of the
  Nagell-Ljunggren equation, the relation being established by point
  \rf{B} of the theorem. The criteria were proved in \cite{Mi1} and
  are in part reminiscent from classical cyclotomic results on
  Fermat's Last Theorem. Thus, the criteria for the First Case, which
  are ennounced in point \rf{C} are the Eichler criterion CF \rf{ii}
  and the criteria of Wieferich and Furtw\"angler (cf. Theorem 2 in \cite{Mi1}). For the Second Case
  of Diagonal Nagell-Ljunggren, in point \rf{C}, it was possible to
  restrict the two conditions proved by Kummer for the FLTII to the
  single condition CF \rf{i}, namely Vandiver's conjecture (cf. Theorem 4 in \cite{Mi1}). This is a
  consequence of the fact that unlike FLT, Nagell-Ljunggren is a
  binary equation, a fact which allowed also to prove upper bounds for
  the solutions, which are given in Theorem \ref{ubounds} below. The
  fact that the Nagell-Ljunggren equation is not homogenous in $X$
  makes it difficult to prove lower bounds, thus leaving a gap on the
  way to a complete proof of Conjecture \ref{bt}.

  The proofs in \cite{Mi1} used methods that generalize the ones that
  helped proving the Catalan Conjecture \cite{Mi2}. A variant of these
  methods will be applied for proving point \rf{B}. We
  gathered the occasion of writing this paper to give in the 
Appendix an extensive exposition of the computations on which the singular case
in the proof the respective estimate from \cite{Mi1} relies: some colleagues had
pointed out that they could not verify the computation on base of the arguments
in \cite{Mi1}, so this difficulty should be dealt with in the Appendix.
\end{remark}

The plan of the paper is as follows: in Section 2 we establish the
connection between equations \rf{bin} and \rf{tbin}, review some basic
properties of Stickelberger ideals and prove auxiliary technical
lemmata concerning coefficients of binomial series development.

With these prerequisites, we complete the proof of Theorem \ref{tbin}
in Section 3. Given the reduction to the Nagell-Ljunggren Diagonal
Case, the proof focuses on point \rf{B} of Theorem \ref{tbin}. In Section
4 we drop the condition that $n$ be a prime and use the proven facts
in order to deduce the results on \rf{bin} for arbitrary exponents $n$
which are stated in Theorem \ref{genexpo}. Finally, the Appendix
provides the details for an estimate used in \cite{Mi1}, as mentioned
in Remark \ref{remark}.

\section{Preliminary results}

The proof of Theorem \ref{tbin} emerges by relating the equation \rf{bin}
to the Diagonal Case of the Nagell -- Ljunggren conjecture. In this section
we shall recall this conjecture and several technical tools used for 
reducing one conjecture to the other. The reduction is performed in the
next section.

\subsection{Link of \rf{bin} with the diagonal Nagell -- Ljunggren
  equation}
We note that $\delta = \left(\frac{X^n-1}{X-1}, X-1\right)$ divides $ n$
and $\delta = n$ exactly when $X \equiv 1 \bmod n$. Indeed,
from the expansion
\[ \frac{X^n-1}{X-1} = \frac{((X-1)+1)^n - 1}{X - 1} = n + k (X-1), \]
with $k \in \Z$, one deduces the claim $\delta \big \vert n$. If $ u
\neq 1$, then $\delta=n$ and thus $n | (X - 1)$ must hold. Conversely,
inserting $X \equiv 1 \bmod n$ in the previous expression shows that
in this case $\delta = n$.

We first show that any solution of \rf{bin} leads to a solution of
\rf{dn}. For this, let $\prod_{i=1}^k
p_i$ be the radical of $\frac{X^n-1}{n^e(X-1)}$. Obviously, $\rad(\frac{X^n-1}{n^e(X-1)}) \ | \ \rad(X^n -
1)$. Let $\zeta \in \C$ be a primitive \nth{n} root of unity. Then the
numbers $\alpha_c = \frac{X - \zeta^c}{(1-\zeta^c)^e} \in \Z[\zeta]$
by definition of $e$, and $(\alpha_c, n) = 1$. Since for distinct $c, d
\not \equiv 0 \bmod n$ we have $( 1-\zeta^d)^e \cdot \alpha_d - (
1-\zeta^c)^e \cdot \alpha_c = \zeta^c - \zeta^d$, it follows that
$(\alpha_c, \alpha_d) \ | \ (1-\zeta)$ and in view of $(\alpha_c, n) =
1$, it follows that the $\alpha_c$ are coprime.

Let $F = \prod_{c=1}^{n-1} \alpha_c = \frac{X^n-1}{n^e(X-1)}$ and $q \ | \ F$ be a rational prime.  In the
field $\Q[ \zeta ]$, it splits completely in the prime ideals $\eu{Q}_c
= (q, \alpha_c), \ c = 1, 2, \ldots, n-1$: these ideals are coprime,
as a consequence of the coprimality of the $\alpha_c$.  Therefore $q
\equiv 1 \bmod n$ and it follows from \rf{nosplit} that $(q, B) = 1$,
so $q | Z$.  Furthermore, \rf{bin} implies that there exists $j_q >0$
such that $q^{j_q n} || Z^n$ and thus $q^{j_q n} || F$. This holds
for all primes $q \ | \ \rad(F)$. It follows that \rf{dn} is verified
for $Y = \prod_{q | F} q^{j_q}$ and $Y \ | \ Z$. We have thus proved
that if $(X, Z)$ is a solution of \rf{bin} for the prime $n$, then
there exists $C \in \Z$ such that $Z = C \cdot Y$ with $Y$ as above,
and:

\begin{eqnarray}
\label{x}
\frac{X^n-1}{n^e (X-1)} & = & Y^n \quad \hbox{ and } \quad \\
\label{y} X-1 & = & B \cdot C^n / n^e.
\end{eqnarray}
We shall write from now on $D = X-1$.

From the above, we conclude that any integer solution of \rf{bin}
induces one of \rf{dn}.  Conversely, if $(X, Y)$ is a solution of
\rf{dn}, then $\left(X, Y; n^e(X-1)\right)$ is a solution of
\rf{bin}. For instance, the particular solution $(X, Y; B) = (18, 7;
17)$ of \rf{bin} stems from
\begin{eqnarray*}
\label{part}
\frac{18^3-1}{18-1} = 7^3,
\end{eqnarray*}
which is supposed to be the only non trivial solution of \rf{dn}. 

\begin{remark}
\label{minus}
Note that if $(X, Z)$ verify \rf{bin}, then $(-X, Z)$ is a solution of
$\frac{X^n + 1}{X + 1} = B Z^n$, so the results apply also to the
equation:
\begin{eqnarray*}
\label{tbins}
X^n + 1 = B Z^n.
\end{eqnarray*}
\end{remark}

\subsection{Bounds to the solutions of Equation \rf{dn}}
We shall use the following Theorem from \cite{Mi1}:
\begin{theorem}
\label{ubounds}
Suppose that $X, Y$ are integers verifying \rf{dn} with $n \geq 17$
being a prime. Let $u = (X \bmod n)$. Then there is an $E \in \R_+$
such that $| X | < E$. The values of $E$ in the various cases of the
equation are the following:
\begin{eqnarray}
\label{Bvals}
E = \begin{cases} \ 4 \cdot \left(\frac{n-3}{2}\right)^{\frac{n+2}{2}}
& \hbox{if $u \not \in \{-1, 0, 1\}$} \\ \ (4 n)^{\frac{n-1}{2}} &
\hbox{if $ u = 0$ }, \\ \ 4 \cdot
\left(n-2\right)^{n} & \hbox{otherwise}.
\end{cases}
\end{eqnarray}
\end{theorem}
By comparing the bounds \rf{Bvals} with \rf{y}, it follows that $| C |
< 2n-1$. In particular, the primes dividing $C$ do not split
completely in $\Q[\zeta_n]$ -- since a prime splitting in this field
has the form $r = 2 k n + 1 > 2n$.
\begin{remark}
\label{4}
Note that $| C | < 2 n-1$ implies a fortiori that for all primes $r |
C$, $r^2 \not \equiv 1 \bmod n$. If $d(r) \leq \Gal(\Q[ \zeta ]/\Q)$
is the decomposition group of the unramified prime $r$, it follows
that $| d(r) | \geq 3$; moreover, either $d(r)$ contains a subcycle
$d' \subset d(r)$ of odd order $| d' | \geq 3$ or it is a cyclic $2$-group 
with at least $4$ elements.
\end{remark}

\subsection{A combinatorial lemma}

\begin{lemma}
\label{sm}
Let $p$ be an odd prime, $k \in \N$ with $1 < k < \log_2(p)$ and $P =
\{ 1, 2, \ldots, p-1\}$. If $S = \{ a_1, a_2, \ldots, a_k \} \subset
P$ be a set of numbers coprime to $p$ and such that $a_i \not \equiv
\pm a_j \bmod p$ for $i \neq j$. We set the bound $A = 2 \lceil
p^{1/k} \rceil$; then there are $k$ numbers $b_i \in \Z, \ i = 1, 2,
\ldots, k$, not all zero, with $0 \leq | b_i | \leq A$ and such that
\[ \sum_{i=1}^k a_i b_i \equiv 0 \bmod p. \]
For $k = 2$, we can choose the $b_i$ such that the additional condition 
\[ \sum_{i=1}^2 b_i/a_i \not \equiv 0 \bmod p. \]
holds.
\end{lemma}
\begin{proof}
  Let $T = \{ 1, 2, \ldots, A\} \subset P$.  Consider the functional
  $f : T^k \rightarrow \ZM{p}$ given by
\[ f(\vec{t}) \equiv \sum_{i=1}^k t_i a_i \ \ \bmod p, \quad \hbox{
  with} \quad \vec{t} = (t_1, t_2, \ldots, t_k) \in T^k.\] Since $|
T^k | > p$, by the pigeon hole principle there are two vectors
$\vec{t} \neq \vec{t'}$ such that $f(\vec{t}) \equiv f(\vec{t'}) \bmod
p$. Let $b_i = t_i - t'_i$; by construction, $0 \leq | b_i | \leq A$
and not all $b_i$ are zero, since $\vec{t} \neq \vec{t'}$. The choice
of these vectors implies $\sum_{i=1}^k a_i b_i \equiv 0 \bmod p$, as
claimed.

We now turn to the second claim. If the claim were false, then 
\[ a_1 b_1 + a_2 b_2 = 0 \hbox{ and } b_1/a_1 + b_2 / a_2 = 0, \]
a homogenous linear system $S$ with determinant 
$\det(S) = \frac{a_1^2-a_2^2}{a_1 a_2}$, which is non vanishing 
under the premise of the lemma. This would imply that 
the solution $b_1, b_2$ is trivial, in contradiction with our construction.
This completes the proof.
\end{proof}

\subsection{Some notation}
We assume that $n$ is prime and let $\zeta$ be a primitive \nth{n}
root of unity, $\K = \Q(\zeta)$ the \nth{n} cyclotomic field and $G =
\Gal(\K/\Q)$ the Galois group. The automorphisms $\sigma_a \in G$ are
given by $\zeta \mapsto \zeta^a, \ a = 1, 2, \ldots, n-1$; complex
conjugation is denoted by $\jmath \in \Z[ G ]$. In the ring of
integers $\Z [\zeta]$, one has finite \textit{$\lambda$-adic} expansions:
for any $\alpha \in \Z[ \zeta ]$ and some $N > 0$ there are 
$a_j  \in \{ -(p-1)/2, \cdots, 0, 1, \cdots, (p-3)/2 \}, j = 0, 1, \ldots N$ such that:
\begin{equation}
\label{lambdaadic}
\alpha = \sum_{j=0}^{N} a_j (1-\zeta)^{j}. 
\end{equation}

We shall use the algebraic $O( \cdot )$-notation, to suggest the
remainder of a power series. This occurs explicitly in the following
four contexts
\begin{itemize}
\item[ (i) ] In a $\lambda$-adic development of the type
  \rf{lambdaadic}, we write $\alpha = x+ O(\lambda^m)$ to mean that
  there is some $y \in \Z[ \zeta ]$ such that $\alpha - x = \lambda^m
  y$. Since $(n) = (\lambda^{p-1})$, powers of $n$ can occur as well
  as powers of $\lambda$ in this notation.
\item[ (ii) ] We also use formal power series, often written $f =
  f(D) \in \K[[ D ]]$.  For $f = \sum_{k=0}^{\infty} f_k D^k$ with
  partial sum $S_m(f) = \sum_{k=0}^{m} f_k D^k$ we may also use the
  $O( \cdot )$-notation and denote the remainder by $f(D) = S_m(D) +
  O( D^{m+1} )$.
\item[ (iii) ] Suppose that $D$ is an integer and the formal power
  series converges in the completion $\K_{\eu{P}}$ at some prime
  $\eu{P} \subset \id{O}(\K)$ dividing $D$. Suppose also that in this 
case all coefficients of $f$ be integral: then the remainder $f(D) -
  S_m(D)$ is by definition divisible by $\eu{P}^{m+1}$, so $O( D^{m+1}
  )$ means in this context that the remainder is divisible by
  $\eu{P}^{m+1}$.
\item[ (iv) ] If $f(D)$ converges at all the prime ideals dividing
  some integer $a | D $, then $O( D^{m+1} )$ will denote a number
  divisible by $a^{m+1}$. In this paper we shall use this fact in the
  context in which $a = p$ is an integer prime dividing $D$ and such
  that $f(D)$ converges at all prime ideals of $\K$ above $p$.
\end{itemize}
\subsection{Auxiliary facts on the Stickelberger module}
\label{SecStick}
The following results are deduced in \cite{Mi1}, Section 4, but see also
\cite{Mi2}, \S 2.1-2.3 and \S 4.1. The results shall only be
mentioned here without proof.

The Stickelberger module is $I = (\vartheta \cdot \Z[ G ])
\cap \Z[ G ]$, where $\vartheta = \frac{1}{n} \sum_{c=1}^{n-1} c \cdot
\sigma_c^{-1}$ is the Stickelberger element. For $\theta = \sum_c
n_c \sigma_c \in I$ we have the relation $\theta + \jmath \theta =
\varsigma(\theta) \cdot \Norm$, where $\varsigma(\theta) \in \Z$ is
called the \textit{relative weight} of $\theta$. The augmentation of
$\theta$ is then
\[ | \theta | = \sum_c n_c = \varsigma(\theta) \cdot \frac{n-1}{2}.\]

The Fueter elements are
\[ \psi_k = (1 + \sigma_{k} - \sigma_{k+1}) \cdot \vartheta =
\sum_{c=1}^{n-1} \left( \left[ \frac{(k+1)c}{n} \right] - \left[
    \frac{kc}{n} \right] \right) \cdot \sigma_c^{-1}, \quad 1 \leq k
\leq (n-1)/2.\] Together with the norm, they generate $I$ as a $\Z$ -
module (of rank $(n+1)/2$) and $\varsigma(\psi_k) = 1$ for all $k$.

The Fuchsian elements are \[ \Theta_k = (k - \sigma_k) \cdot \vartheta
= \sum_{c=1}^{n-1} \left[ \frac{kc}{n} \right] \cdot \sigma_c^{-1},
\quad 2 \leq k \leq n.\] They also generate $I$ as a $\Z$ -
module. Note that $\Theta_n$ is the norm, and that we have the
following relationship between the Fueter and the Fuchsian elements:

\begin{equation*}
\psi_1  =  \Theta_2 \text{   and   } 
\psi_k  =  \Theta_{k+1} - \Theta_k,\ k \geq 2
\end{equation*}

An element $\Theta = \sum_c n_c \sigma_c$ is \textit{positive}
if $n_c \geq 0$ for all $c \in \{ 1, 2, \ldots, p-1\}$. We write
$I^+ \subset I$ for the set of all positive elements. They 
form a multiplicative and an additive semigroup.
  
The Fermat quotient map $I \rightarrow \ZM{n}$, given by $$\varphi :
\theta = \sum_{c=1}^{n-1} n_c \sigma_c \mapsto \sum_{c=1}^{n-1}
c n_c \bmod n,$$ is a linear functional, with kernel $I_f=\{ \theta
\in I : \zeta^{\theta} = 1\}$ (the \textit{Fermat module}), and enjoys
the properties:
\begin{eqnarray*}
\label{fc}
\zeta^{\theta} & = & \zeta^{\varphi(\theta)}, \nonumber \\
(1+\zeta)^{\theta} & = & \zeta^{\varphi(\theta)/2}, \\
(1-\zeta)^{\theta} & = & \zeta^{\varphi(\theta)/2} \cdot \left(
\lchooses{-1}{n} n\right)^{\varsigma(\theta)/2}, \nonumber
\end{eqnarray*}
where $\lchooses{-1}{n} $ is the Legendre symbol.

The last relation holds up to a sign which depends on the embedding of
$\zeta$. For a fixed embedding, we let $\nu = \sqrt{\lchooses{-1}{n}
  n}$ be the generator of the quadratic subfield $\Q(\nu) \subset \Q(\zeta)$. 
A short computation shows that $(1-\zeta)^{\theta} = \zeta^{\varphi(\theta)/2}
\nu$.  Note that for $\theta \in I$ with $\varsigma(\theta) = 2$ we
have $(1-\zeta)^{2 \theta} = \zeta^{\varphi(\theta)} \cdot n^2$ for
any embedding.

We shall want to consider the action of elements of $\theta \in \F_n[G]$ on explicit algebraic
numbers $\beta \in \K$. Unless otherwise specified, an element $\theta = \sum_{c=1}^{n-1} m_c \sigma_c \in \F_n[G]$ is lifted to $\sum_{c=1}^{n-1} n_c \sigma_c$, where $n_c \in \Z$ are the unique integers with $0 \leq n_c < p$ and $n_c \equiv m_c \mod p$. In particular, lifts are always positive, of bounded weight $w(\theta) \leq (p - 1)^2$. Rather than introducing an additional notation for the lift defined herewith, we shall always assume, unless otherwise specified, that $\theta \in \F_p[G]$ acts upon $\beta \in \K$ via this lift.

Using this lift, we define the following additive maps: 
\begin{equation*}
\rho_0 :  \F_n [G]  \ra \Q(\zeta) \quad
  \theta = \sum_{c=1}^{n-1} n_c \sigma_c  \mapsto  \sum_{c \in P} \frac{n_c}{1- \zeta^c},
\end{equation*}
and
\begin{equation*}
\rho :  \F_n [G]  \ra \Z[\zeta] \quad
 \theta  \mapsto  (1-\zeta) \cdot \rho_0[\theta].
\end{equation*}
The \nth{i} moment of an element $\theta = \sum_{c=1}^{n-1} n_c
\sigma_c$ of $\Z[G]$ is defined as:
\begin{equation*}
\phi^{(i)} (\theta) = \sum_{c=1}^{n-1} n_c c^i \mod n.
\end{equation*}
Note that $\phi^{(1)}$ is the {\em Fermat quotient map}: $\phi^{(1)} =
\varphi$.  The moments are linear maps of $\F_p$-vector spaces and
homomorphism of algebras, verifying:
\begin{eqnarray}
\label{phiprop}
\begin{array}{l c l l c l}
  \qquad \phi^{(i)}(a \theta_1 + b \theta_2) & = & a \phi^{(i)}(\theta_1) + b \phi^{(i)}(\theta_2), & \hbox{and} & \\
 \qquad \phi^{(i)}(\theta_1 \theta_2) & = & \phi^{(i)}(\theta_1) \phi^{(i)}(\theta_2), & \hbox{with } & \theta_j \in \F_p[ G ]; a, b \in \F_p.
\end{array}
\end{eqnarray}
The linearity in the first identity is a straight-forward verification
from the definition.  For the second, note that for $\theta = \sum_c
n_c \sigma_c$ we have
\[ \phi^{(i)}(\sigma_a \theta) = \phi^{(i)}\left(\sum_c n_c \sigma_{a c}\right) = 
\sum n_c \cdot (a c)^i = a^i \cdot \phi^{(i)}(\theta). \] 
Using the already established linearity, one deduces the multiplicativity of  
$\phi^{(i)}$ as a ring homomorphism.

Let $\alpha = \frac{X-\zeta}{(1-\zeta)^e} \in \Z[ \zeta]$, as before,
and define $c_X \equiv 1/(X-1) \bmod n$ if $e = 0$ and $c_X = 0$ if $e
= 1$. For any $\theta \in I^+$, there is a \textit{Jacobi integer}
$\beta[\theta] \in \Z[ \zeta ]$ such that $\beta[\theta]^n =
(\zeta^{c_X} \alpha)^{\theta}$, normed by $\beta[\theta] \equiv 1
\bmod (1-\zeta)^2$ (Lemma 2 of \cite{Mi1}). The definition of $\varsigma(\theta)$
implies that
\begin{eqnarray}
\label{norm}
\beta[\theta] \cdot \overline{\beta[\theta]} =
\Norm_{\K/\Q}(\alpha)^{\varsigma(\theta)} = Y^{\varsigma(\theta)}.
\end{eqnarray}

We have for any $\theta \in I^+$,
\begin{eqnarray}
\label{beta1}
\beta[\theta]^n = (\zeta^{c_X} \alpha)^{\theta} = (\zeta^{c_X}
(1-\zeta)^{1-e})^{\theta} \cdot \left(1 + \frac{X-1}{1-\zeta}\right)^{\theta}
\end{eqnarray}

\begin{lemma}
\label{theta}
We remind that $D=X-1$. For any $\theta \in 2 \cdot I_f^+$, for any prime ideal $\eu{P} \ | \
D$, there is a $\kappa = \kappa_{\eu{P}}(\theta) \in \ZM{n}$ such that
\begin{equation*}
\label{base1} 
\beta[ \theta ] \equiv \zeta^{\kappa} \cdot Y^{\frac{\varsigma(\theta)}{2}} \bmod \eu{P}.
\end{equation*}
\end{lemma}

\begin{proof}
  Let $\theta_0$ be an element of $I^+_f$, and let $\theta = 2
  \theta_0$. Note that from \rf{norm} we have $Y^{\varsigma(\theta_0)
    n} = \beta[\theta_0]^n \cdot \overline \beta[ \theta_0 ]^n$. Thus
  $ \beta[\theta]^n = \beta[\theta_0]^{2 n} = Y^{\varsigma(\theta_0)
    n} \cdot \left ( \beta[\theta_0]/\overline \beta[\theta_0]
  \right)^{n}$.  Using \rf{beta1} and the previous observations, we
  find:
\begin{eqnarray}
  \beta[\theta]^n & = & Y^{\varsigma(\theta_0)n} \cdot \left(\zeta^{c_X} \cdot
    (1-\zeta)^{1-e}\right)^{ (\theta_0 - \jmath \theta_0)} \cdot \left(1 +
    \frac{X-1}{1-\zeta}\right)^{(\theta_0 - \jmath \theta_0)} \nonumber \\
  & = & Y^{\varsigma(\theta_0) n} \cdot \zeta^{(2c_X + 1) \varphi(\theta_0)} \cdot \left(1 + D/(1-\zeta)\right)^{ (\theta_0 - \jmath
    \theta_0)} \nonumber \\
  \beta[\theta]^n & = & Y^{\varsigma(\theta_0) n} \cdot \left(1 + \frac{D}{1-\zeta}\right)^{(\theta_0 - \jmath
    \theta_0)}. \label{pow} 
\end{eqnarray}
Thus for any prime ideal $\eu{P} \ | \ D$ there is a $\kappa =
\kappa_{\eu{P}}(\theta) \in \ZM{n}$ such that
\begin{equation}
\label{kappa} 
  \beta[ \theta ]  \equiv  \zeta^{\kappa} \cdot Y^{\varsigma(\theta_0)} \bmod \eu{P}. 
\end{equation}
\end{proof}
In the sequel, we indicate how to choose $\theta$ such that
$\kappa=0$. In this case, the relation \rf{beta1} leads to a
$\eu{P}$-adic binomial series expansion for $\beta[\theta]$.

We will use the Voronoi identities -- see Lemma 1.0 in \cite{Jha}
--, which we remind here for convenience:

\begin{lemma}
\label{Jha}
Let $m$ be an even integer such that $2 \leq m \leq n-1$. Let $a$ be
an integer, coprime to $n$. Then

\begin{equation}
\label{voronoi}
  a^m \sum_{j=1}^{n-1} \left[ \frac{aj}{n} \right] j^{m-1} \equiv \frac{(a^{m+1}-a)B_m}{m} \mod n,
\end{equation}
where $B_m$ is the $m$-th Bernoulli number. In particular, for
$m=n-1$, we get

\begin{equation*}
  \sum_{j=1}^{n-1} \left[ \frac{aj}{n} \right] j^{n-2} \equiv \frac{a^{n}-a}{n} \mod n,
\end{equation*}
which is the Fermat quotient map of the $a$-th Fuchsian element,
$\varphi(\Theta_a)$.
\end{lemma}

\begin{lemma}
\label{simple}
Let $\psi_k$ denote the $k$-th Fueter element. Then, there exists a
linear combination $\theta = \sigma \psi_k + \tau \psi_l \in I$ with
$\sigma, \tau \in G$ and $1 \leq k , l < n$ such that $\phi^{(1)}
(\theta) = 0$ and $\phi^{(-1)} (\theta) \neq 0$.
\end{lemma}
The proof of this Lemma is elementary, using the Voronoi relations
\rf{voronoi}; since the details are rather lengthy, they will
be given in the Appendix.

The following two lemmata contain computational information for binomial
series developments that we shall use below. First, we remind that $\rho_0$ is the following additive map:
\begin{equation*}
\rho_0 :  \F_n [G]  \ra \Q(\zeta) \quad
  \theta = \sum_{c=1}^{n-1} n_c \sigma_c  \mapsto  \sum_{c \in P} \frac{n_c}{1- \zeta^c}
\end{equation*}

\begin{lemma}
\label{dvpt}
Let $D$ be an indeterminate. Let $\theta = \sum_{c=1}^{n-1} n_c \sigma_c \in \Z[G]$ and $f[\theta]
= \left( 1 + \frac{D}{1-\zeta} \right)^{\theta/n} \in \K[[ D ]]$. 
Let $0 < N < n$ be a fixed integer. Then, 
\begin{equation*}
f[\theta] = 1 + \sum_{k=1}^N \frac{a_k[\theta]}{k! n^k} D^k + O(D^{N+1}),
\end{equation*}
where, for $1 \leq k \leq N$, we have
\begin{equation*}
a_k[\theta] = \rho_0^k [\theta] + O\left( \frac{n}{(1-\zeta)^k}\right).
\end{equation*}
In the above identity, $a_k[ \theta ], \rho_0^k[ \theta ] \in \Z[
\zeta, \frac{1}{n} ]$ are not integral, but their difference is an
algebraic integer $a_k[ \theta ] - \rho_0^k[ \theta ] \in
\frac{n}{(1-\zeta)^k} \cdot \Z[ \zeta ]$.
\end{lemma} 

\begin{proof}
  Let $\theta = \sum_c n_c \sigma_c$ and $m = m(\theta) = | \ \{ c \ :
  \ n_c \neq 0 \} \ |$ be the number of non vanishing coefficients of
  $\theta$.  We prove this result by induction on $m$. First, note
  that
\begin{equation*}
\binom{n_c/n}{k} = \frac{1}{k!} \cdot \frac{n_c^k}{n^k} \cdot (1 + O(n)).
\end{equation*}
Thus, if $\theta = n_c \sigma_c$ and $m = 1$, then:
\begin{equation*}
  f[\theta] = 1 + \sum_{k=1}^{n-1} \frac{1}{k!} \cdot \frac{n_c^k}{n^k} \cdot (1 + O(n)) \cdot \frac{D^k}{(1-\zeta)^k} = 1 + \sum_{k=1}^N \frac{a_k[\theta]}{k! n^k} D^k + O(D^{N+1}),
\end{equation*}
where, for $1 \leq k \leq N$,
\begin{equation*}
a_k[\theta] = \rho_0^k [\theta] + O\left( \frac{n}{(1-\zeta)^k}\right),
\end{equation*}
which confirms the claim for $m = 1$.  Suppose the claim holds for all
$j \leq m$ and let $\theta = \theta_1 + \theta_2$ with $m(\theta_i) <
m$ and $m(\theta) = m$. Then,
\begin{equation*}
\begin{array}{lcl}
f[\theta] & = & \left( 1 + \frac{D}{1-\zeta} \right)^{\theta_1/n} \cdot \left( 1 + \frac{D}{1-\zeta} \right)^{\theta_2/n}\\
& = & 1 + \sum_{k=1}^N \alpha_k[\theta] D^k + O(D^{N+1}),
\end{array}
\end{equation*}
where for $k < n-1$ we have
\begin{equation*}
\begin{array}{lcl}
  \alpha_k[\theta] & = & \sum_{j=1}^k \frac{a_j [\theta_1]}{n^j j! (1-\zeta)^l} \cdot \frac{a_{k-j}[\theta_2]}{n^{k-j} (k-j)! (1-\zeta)^{k-l}} \cdot (1 + O(n))\\
  & = & \frac{1}{k!n^k}  \left(  \rho_0 [\theta_1] + \rho_0 [\theta_2] \right)^{k} + O\left( \frac{n}{k! n^k(1-\zeta)^k}\right)\\
  & = & \frac{1}{k!n^k} \cdot \rho_0^k [\theta] + O\left( \frac{n}{k! n^k(1-\zeta)^k}\right) = 
  \frac{1}{k!n^k} \cdot \left( \rho_0^k [\theta] + O( n/(1-\zeta)^k\right)
\end{array}
\end{equation*}
This proves the claim by complete induction.
\end{proof}

\begin{lemma}
\label{RemInt}
By proceeding like in Lema 8 of \cite{Mi2}, we notice that $\frac{a_k[\theta]}{k!} \in \Z[\zeta]$ (notation is different between both articles).
\end{lemma}

As a consequence, we may deduce that matrices built from the first
coefficients occurring in some binary series developments are regular.
\begin{lemma}
\label{regular}
Let $\theta = \sum_{c=1}^{n-1} n_c \sigma_c \in \Z[G]$ such that
$\phi^{(-1)}(\theta) \not\equiv 0 \mod n$, let $f[\theta] = \left( 1 +
  \frac{D}{1-\zeta} \right)^{\theta/n}$ and $0<N<n-1$ be a fixed integer. 
Then, 
\[ f[\theta] = 1 + \sum_{k=1}^N \frac{b_k[\theta]}{k! n^k (1-\zeta)^k}
D^k + O(D^{N+1}), \quad \hbox{with} \quad \frac{b_k[\theta]}{k!} \in
\Z[\zeta]. \] Moreover, if $J \subset \Gal(\Q[ \zeta ]/\Q)$ is a
subset with $| J | = N$, then the matrix\footnote{We shall apply this 
Lemma below, in a context in which $J$ satisfies the additional
condition that $i + j \neq n$ for any $i, j$ with $\sigma_i \in J$ 
and $\sigma_j \in J$.}
\[ A_N = \left(b_k[\sigma_c \theta] \right)_{ k=0;\sigma_c \in J}^{N-1}
\in \GL(\K, N) \].
\end{lemma}
\begin{proof}
  Let $\lambda = 1-\zeta$; we show that the determinant of $A_N$ is
  not zero modulo $\lambda$. Using Lemma \ref{dvpt}, we know that we
  have a development of symbolic power series
\begin{equation*}
f[\theta] = 1 + \sum_{k=1}^N \frac{a_k[\theta]}{k! n^k} D^k + O(D^{N+1}),
\end{equation*}
where
\begin{equation*}
a_k[\theta] = \rho_0^k [\theta] + O\left( \frac{n}{(1-\zeta)^k}\right).
\end{equation*}
By definition, $(1-\zeta)^k \cdot a_{k} [\sigma_c \theta] \in
\Z[\zeta]$ for all $\sigma_c \in G$. Let $b_k[\theta] = (1-\zeta)^k
\cdot a_{k} [\theta] \in \Z[\zeta]$. Then, according to Lemma
\ref{dvpt},

\begin{equation*}
\begin{array}{rcl}
  b_k[\sigma_c \theta]  & = & (1-\zeta)^k \cdot \left( \rho_0^k [\sigma_c \theta] + O\left( \frac{n}{(1-\zeta)^k}\right) \right)\\
  & = & \rho^k [\sigma_c \theta] + O(n)  =  \left( \sum_{l=1}^{n-1} n_l \cdot \frac{1-\zeta}{1-\zeta^{lc}} \right)^k + O(n) \\
  &  \equiv &  \left( \sum_{l=1}^{n-1} \frac{n_l}{lc} \right)^k \mod \lambda 
  \equiv \left( \frac{\phi^{(-1)} [\theta]}{c} \right)^k \mod \lambda.
\end{array}
\end{equation*}
Thus, $\det A_N \equiv \left| \left( \left( \frac{\phi^{(-1)}
        [\theta]}{c} \right)^k\right)_{k=0,\sigma_c \in J}^{N-1} \right|
\mod \lambda$. We have obtained a Vandermonde determinant:
\begin{equation*}
\det A_N \equiv \left( \phi^{(-1)} [\theta] \right)^{N(N-1)/2} \cdot \prod_{i \neq j; \sigma_i, \sigma_j \in J} \left( \frac{1}{i} - \frac{1}{j}\right) \mod \lambda.
\end{equation*}
By hypothesis, $\phi^{(-1)} [\theta] \not\equiv 0 \mod n$,
and $1/i \not \equiv 1/j \bmod n$ for $\sigma_i, \sigma_j \in J$;
this implies finally that $\prod_{\sigma_i, \sigma_j \in J} \left(
  \frac{1}{i} - \frac{1}{j}\right) \not\equiv 0 \mod n$, which
confirms our claim.
\end{proof}

\section{Proof of Theorem \ref{tbin}}

Theorem 4 in \cite{Mi1} proves that if CF holds, then \rf{dn} has no solution except for \rf{part}. The computations in \cite{BH} prove that CF holds for $n\leq 163.10^6$. This proves Theorem \ref{tbin}.\rf{A}.
Theorem \ref{tbin}.\rf{C} is also proved in Theorem 4 in \cite{Mi1}. In the sequel we shall show that the
only possible solutions are $X = \pm B/n^e + 1$. We may assume in particular that $n > 163 \cdot 10^6$.

We have already proved that $X-1 = B \cdot C^n/n^e$. If $C = \pm 1$,
then $X - 1 = \pm B/n^e$, as stated in point \rf{B} of Theorem \ref{tbin} and
$X$ is a solution of \rf{dn}. The bounds on $| X |$ in \rf{Bvals}
imply $| B | < n^n$, the second claim of \rf{B}.

Consequently, Theorem \ref{bin} will follow if we prove that $C = \pm
1$; we do this in this section. Assume that there is a prime $p | C$ with $p^{i} || C$. 
Let $\eu{P} \subset \Z[\zeta]$ be a prime ideal lying
above $p$ and let $d(p) \subset G$ be its decomposition group. We shall use
Remark \ref{4} in order to derive some group ring elements which cancel 
the exponents $\kappa$ occurring in \rf{kappa}. 

Recall that $D = B \cdot C^n/n^e = X-1$, with $C$ defined by
\rf{y}. Note that \rf{y} implies that either $(n, D) = 1$, or $n^2 |
B$ and $(n, C) = 1$. Indeed, if $(n, D) \neq 1$, then $e = 1$ and $n^2
| n^e (X-1) = B C^n$ and since $(C, n) = 1$, it follows that $n^2 |
B$; the last relation follows from the bounds $C^n \leq E < 4
(n-2)^n$, hence $| C | < n$. In both cases $1/(1 - \zeta)$ is
congruent to an algebraic integer modulo $D/n^{v_n(D)} \cdot
\Z[\zeta]$.

According to Remark \ref{4}, we know that there are at least two elements, $\sigma^{'}_1, \sigma^{'}_2 \in
d(p)$ such that $\sigma^{'}_1 \neq \jmath \cdot \sigma^{'}_2$. Let
$\sigma^{'}_i(\zeta) = \zeta^{c_i}, \ c_i \in \ZMs{n}$. It
follows from Lemma \ref{sm} that, for $c_i \neq c_j$ when $i \neq j$, there are $h'_1, h'_2 \in \Z$ with $|
h'_i | \leq \sqrt{n}$ and $\sum_{i=1}^2 h'_i c_i \equiv 0 \bmod n$ while
$\sum_{i=1}^2 h'_i / c_i \not \equiv 0 \bmod n$.

We define 
\begin{eqnarray}
\label{mudef}
\begin{array}{l c l}
\mu & = & \sum_{i=1}^2 h_i \sigma_i \in \Z[ d(p) ] \subset \Z[ G ], \quad \hbox{with} \\
 h_i & = & \begin{cases} 
		h'_i & \hbox{if $h'_i > 0$ and }\\
		- h'_i & \hbox{otherwise}, \quad \quad \quad \hbox{and}
          \end{cases} \\
\sigma_i & = & \begin{cases} 
		\sigma^{'}_i & \hbox{if $h'_i > 0$ and }\\
		\jmath \sigma^{'}_i & \hbox{otherwise}.
          \end{cases}
\end{array}
\end{eqnarray}
By construction, $\mu$ is a positive element, i.e. the coefficients
$h_i \geq 0$.  Let $\widehat{ \cdot } : G \ra \ZMs{n}$ denote the
cyclotomic character and note that $h'_i \widehat{\sigma} = h_i
\widehat{\sigma'}$ for $h'_i < 0$ and thus $\phi^{(1)}(\mu) =
0$. In view of Lemma \ref{sm}, we also know that we can choose the
$h'_i$ and thus $\mu$, such that
\[ \phi^{(1)}(\mu) = 0, \quad \hbox{ but } \quad \phi^{(-1)}(\mu) \neq 0. \]

Since $\K/\Q$ is abelian, all the primes $\eu{P} | (p)$ have the same
decomposition group $d(p)$ and $\mu$ enjoys the following stronger
property: let $\eu{P} | (p)$ and $S \subset G$ be a set of
representatives of $G/d(p)$; let $\gamma \in \Z[\zeta]$ be such that
$\gamma \equiv \zeta^{c_{\sigma}} \bmod \sigma (\eu{P} )$ for all
$\sigma \in S$; then $\gamma^{\mu} \equiv 1 \bmod p \Z[\zeta]$, as
follows directly from $\zeta^{\mu} \equiv 1 \bmod \sigma(\eu{P})$, for
all $\sigma \in S$.

In view of Lemma \ref{simple} and the fact that Fueter elements are positive, 
we also know that there is a $\theta_0 \in I_f^+$ such that 
$\varsigma(\theta_0) = 2$ and $\phi^{(-1)} (\theta_0) \neq 0$.

Let
\begin{equation*}
  \Theta  =  2 \cdot \mu \cdot \theta_0.
\end{equation*}

In view of the properties \rf{phiprop} of moments and since for both 
$\mu, \theta_0$, the Fermat quotient vanishes, while $\phi^{(-1)}$ 
is non-null, it follows that the same must hold for $\Theta$, 
so $\Theta \in 2 \cdot I_f^+$ and $\phi^{(-1)} (\Theta)
\neq 0$. Let 
\[ h = 2 \cdot \sum_{i=1}^l |h_i| = 2 . w( \mu ) , \]
where we defined the \textit{absolute weight} $w(\sum_c n_c \sigma_c)
= \sum_c | n_c |$. From subsection \ref{SecStick}, we know that there
exists a Jacobi integer $\beta[ 2 \theta_0 ] \in \Z[\zeta]$ such that
$\beta[ 2 \theta_0 ]^n = (\zeta^{c_X} (1-\zeta)^{1-e})^{\theta_0} \cdot
\left(1 + \frac{X-1}{1-\zeta}\right)^{\theta_0}$ (see \rf{beta1}). It
follows from \rf{base1} and Lemma \ref{theta}, that in both cases we
have $\beta[ 2 \theta_0 ] \equiv \zeta^{\kappa(\theta_0)} \cdot Y^4 
\bmod \eu{P}$. We have chosen $\mu$ as a linear combination of two
elements from the decompostion group $D(\eu{P}) \subset G$, so $\mu$
acts on $\zeta \bmod \eu{P}$ by $\zeta \mod \eu{P} \mapsto \zeta^{\mu} \equiv 1 \bmod \eu{P}$.
Therefore, from $\beta[ \Theta ] = \beta[ 2 \theta_0 ]^{\mu}$ and thus, 
by the choice of $\mu$, we have
\begin{eqnarray}
\label{unif}
\beta[ \Theta ] \equiv  Y^h \bmod p \Z[ \zeta ].
\end{eqnarray}
Let $\Theta = 2 \sum_{c = 1}^{n-1} n_c \sigma_c$; for any prime
$\eu{P} | (p)$, the binomial series of the \nth{n} root of the right
hand side in \rf{pow} converges in the $\eu{P}$ - adic valuation and
its sum is equal to $\beta[ \Theta ]$ up to a possible \nth{n} root of
unity $\zeta^c$. Here we make use of the choice of $\Theta$: comparing
\rf{unif} with the product above, it follows that $\zeta^c = 1$ for
all primes $\eu{P} \ | \ (p)$. For any $N > 0$, we have $p^{i n N} | |
D^N$ and thus
\begin{eqnarray}
\label{beta}
  \beta[ \Theta ]  \equiv  Y^h \prod_{c=1}^{n-1} \left(\sum_{k=0}^{N-1}
    \binom{n_c/n}{k} \left(\frac{D}{1-\zeta^c}\right)^k \right) \quad \bmod p^{i n N}. 
\end{eqnarray}

We develop the product in a series, obtaining an expansion which
converges uniformly at primes above $p$ and is Galois covariant; for
$N < n-1$ and $\sigma \in G$, we have:
\begin{eqnarray*}
\label{inhom}
\beta[ \sigma \Theta ] = Y^h \left(1 + \sum_{k = 1}^{N-1}
\frac{b_k[ \sigma \Theta ]}{(1-\zeta)^k n^k  k!} \cdot D^k \right) + O
(p^{i n N}),
\end{eqnarray*}
with $b_k[ \Theta ] \in \Z[\zeta]$.  Let $P \subset \{ 1, 2, \cdots ,
n-1 \}$ be a set of cardinal $1 < N < (n-1)/2$ such that if $c \in P$ then
$n-c \not \in P$, and $J \subset \Z[G]$ be the Galois automorphisms of
$\K$ indexed by $P$: $J=\{ \sigma_c \}_{c \in P}$. Consider the linear
combination $\Delta = \sum_{\sigma \in J} \lambda_{\sigma} \cdot
\beta[\sigma \cdot \Theta]$ where $\lambda_{\sigma} \in \Q[\zeta]$
verify the linear system:
\begin{eqnarray}
\label{hom}
\sum_{\sigma \in J} & & \lambda_{\sigma} \cdot b_k[\sigma \cdot \Theta] =  0 ,
\text{ for } k = 0, \ldots, N-1,\  k \neq \lceil N/2 \rceil \nonumber
\quad \text{ and } \\ 
\sum_{\sigma \in J} & &  \lambda_{\sigma} \cdot b_{\lceil N/2 \rceil}[\sigma \cdot \Theta]  =  (1-\zeta)^{\lceil N/2 \rceil} n^{\lceil N/2 \rceil} \lceil N/2 \rceil !. 
\end{eqnarray}
Applying Lemma \ref{dvpt} we observe that this system is regular for
any $N < n-1$. There exists therefore a unique solution which is not
null.

We recall that a power series $\sum_{k=0}^{\infty} a_k X^k \in
\C[[X]]$ is dominated by the series $\sum_{k=0}^{\infty} b_k X^k \in
\R[[X]]$ with non-negative coefficients, if for all $k \geq 0$, we
have $|a_k| \leq b_k$. The dominance relation is preserved by addition
and multiplication of power series.

Following the proof of Proposition 8.2.1 in \cite{Bilu2}, one shows
that if $r \in \R_{>0}$ and $\chi \in \C$, with $|\chi| \leq K$ with
$K \in \R_{>0}$, then the binomial series $(1+\chi T)^r$ is dominated
by $(1- K T)^{-r}$. From this, we obtain that $(1+ \chi T)^{\Theta/n}$ is
dominated by $(1 - K T)^{-w(\Theta)/n}$. In our case \rf{beta}, $T=D$, $\chi = \frac{1}{1 - \zeta^c}$ and 
\[ K = \max_{1 \leq c < n} | 1/(1-\zeta^c) | = 1/\sin(\pi/n) \leq
n/\pi \cos(\pi/3) = 2n/\pi < n . \] Applying this to our selected
$\Theta$, whose absolute weight is bounded by $w \leq 4 n \sqrt{n}$,
we find after some computations that $|b_k[\sigma \cdot \Theta]| <
n^{k} \cdot \binom{-w/n}{k} \cdot k! < n^{3k}$ for $N < n/2$.

Let $A = \det \left( b_k[\sigma_c \cdot \Theta ] \right)_{k=0; c \in
  I}^{N-1} \neq 0$ be the determinant of the matrix of the system
\rf{hom}, which is non vanishing, as noticed above: note that the
division by $k!$ along a complete row does not modify the regularity
of the matrix.

Let $\vec{d} = (1-\zeta)^{\lceil N/2 \rceil} n^{\lceil N/2 \rceil} \lceil N/2 \rceil !  \left( \delta_{k,\lceil N/2
    \rceil}\right)_{k=0}^{N-1} $, where $\delta_{i,j}$ is Kronecker's
symbol. The solution to our system is $\lambda_\sigma = A_{\sigma} /
A$, where $A_\sigma \in \Z[\zeta]$ are the determinants of some minors
of $\left( b_k[\sigma_c \cdot \Theta ] \right)_{k=0;c \in I}^{N-1}$
obtained by replacing the respective column by $\vec{d}$.

Noticing that $|(1-\zeta)^{\lceil N/2 \rceil} n^{\lceil N/2 \rceil} \lceil N/2 \rceil !| < n^{3(N-1)}$, Hadamard's inequality implies that
\begin{eqnarray*}
  |A_\sigma| & \leq & n^{3(N-1)(N-2)/2} \cdot (N-1)^{(N-1)/2} \leq n^{3N^2/2} \cdot N^{N/2} \quad \hbox{and} \\
  |A| & \leq & n^{3N^2/2} \cdot  N^{N/2}
\end{eqnarray*}
 Let $\delta = A \cdot \Delta \in \Z[\zeta]$, 
 \[ \delta = \sum_{\sigma \in J} A_{\sigma} \cdot \beta[\sigma \cdot
 \Theta] \in \Z[\zeta].  \]

We set $N = \lceil n^{3/4} \rceil$ and claim that for such $N$, $\delta \neq 0$. By choice of the $\lambda$'s, we have $\delta = A.p^{in \lceil N/2 \rceil}.u  + p^{i n N} z$ for some $z \in \Z[ \zeta ]$, where $u = \frac{D^{\lceil N/2 \rceil}}{ p^{in \lceil N/2 \rceil}}.Y^h$ is a unit in $(\Z/p\Z)^*$. Therefore, if we assume that $\delta = 0$, then necessarily $p^{in \lceil N/2 \rceil}$ divides $A$. However, $v_p(A) < n \lceil N/2 \rceil$. Indeed, the upper bound for $| A |$ implies a fortiori that $v_p(A) \leq \lceil N/2 \cdot \log N + \frac{3 N^2}{2}\log n \rceil$.  Then, the assumption $\delta = 0$ would imply $n \leq 3\left[ n^{3/4} + \frac{1}{4}\right] \log n$, which is false for $n \geq 4,5.10^6$. This contradicts thus our initial assumption. Therefore $\delta \neq 0$.

 Given the bounds on $A_{\sigma}$, we obtain $|\delta| \leq N Y^h
 n^{3N^2/2} \cdot N^{N/2}$ and using the fact that $h < 4n^{1/2}$, $Y
 < n^n$ (Theorem \ref{tbin}.\rf{B}) and $N = \lceil n^{3/4} \rceil$, we find
\begin{eqnarray*}
\label{up}
| \Norm_{\K/\Q}(\delta) | < \left( n^{\frac{11}{2} n^{3/2} + \frac{3}{8} n^{3/4} + 
    \frac{3}{4}}\right) ^{n-1}.
\end{eqnarray*}
The initial homogenous conditions in \rf{hom} imply $\delta \equiv 0
\bmod p^{i n \lceil N/2 \rceil }$, therefore $| \Norm_{\K/\Q}(\delta)
| \geq p^{i n (n-1) N/2}$.  Combining this inequality with \rf{up} and
$n \geq 163 \cdot 10^6$, one finds that $\log p < 1.64$. This shows
that $p= 2, 3$ or $5$.

We consider the case $p \leq 5 $ separately as follows. We choose in this case $\mu = 1 + p \jmath
\sigma_p^{-1}$ and verify that $\varphi(\mu) = 0$, while
$\phi^{-1}(\mu) = 1 - p^2 \not \equiv 0 \bmod n$. Consequently
$\varsigma(\Theta) = 4(p+1)$ and the norm of $\delta$ is thus bounded by
\begin{eqnarray*}
p^{n (n-1) N/2} \leq  | \Norm_{\K/\Q}(\delta) | < \left( n^{4(p+1) + 3N^2/2} \cdot N^{N/2+1} \right)^{n-1}.
\end{eqnarray*}
Letting $N = 48$, we obtain the inequality
\[ 2^n \leq n^{73} \cdot 48^{25/24} < 64 n^{73} \quad \Rightarrow
\quad \frac{n-6}{73} \leq \log(n)/\log(2), \] which is false for for
$n > 695$, and a fortiori for $n > 163.10^{6}$. We obtain a contradiction
in this case too, and thus $C = \pm 1$, which completes the proof of
Theorem \ref{tbin}.
\qed

\section{Consequences for the general case of the binary Thue equation  \rf{bin}}
In this section we derive Theorem \ref{genexpo}. For this we
assume that \rf{bin} has a solution with $(\varphi^*(B),n) = 1$, since
our results only hold in this case, a fact which is reflected also in
the formulation of Theorem \ref{genexpo}.

Consider the case when $n = p \cdot q$ is the product of two distinct
primes. If $(n , B) = 1$, then Theorem \ref{tbin} holds for both $p$
and $q$ with the value $e = 0$. If $X, Z$ is a solution, then Theorem
\ref{tbin} . \rf{B} implies that $X^p = \pm B + 1$ and $X^q = \pm B
+1$. Consequently either $X^p + X^q = 2$ or $X^p - X^q = 2$. This is
impossible for $| X | > 2$ and a simple case distinction implies that
there are no solutions. As a consequence,

\begin{corollary}
\label{c1}
Consider Equation \rf{bin} for fixed $B$ and suppose that $n$ is
an integer which has two distinct prime divisors $p > q > 2$ with $(p, B) =
( q, B ) = 1$. Then \rf{bin} has no solutions for which \rf{nosplit} holds.
\end{corollary}

If all divisors of $n$ are among the primes dividing $B$, we are led
to the following equation: $p (X^q - 1) = q (X^p - 1)$, which has no
solutions in the integers other than $1$. Indeed, assume $X \neq 1$ to be a
solution of the previous equation, and $q=p+t,\ t \geq 0$. The
real function $f(t) = p(X^{p+t}-1) - (p+t)(X^p-1)$ is strictly
monotonous and $f(0)=0$.  Therefore, the equation $p (X^q - 1) = q
(X^p - 1)$ has no solutions. There is only the case left in which $n$
is built from two primes, one dividing $B$ and one not. In this case,
one obtains that equation $p ( X^q - 1) = X^p - 1$ which can also be
shown not to have non trivial solutions, using the above remark, this
time with $f(t) = p(X^{p+t}-1) - (X^p-1)$. Hence:
\begin{corollary}
\label{c2}
The equation \rf{bin} has no solutions for exponents $n$ which are
divisible by more than one prime and for $B$ such that \rf{nosplit} holds.
\end{corollary}

We are left to consider the case of prime powers $n = p^c$ with $c >
1$. If $p \nmid B$, we obtain $X^{n/p} - 1 = B/p^e$, so in particular
$B/p^e+1 \geq 2^{p^{c-1}}$ is a \nth{p^{c-1}} power. Since in this
case, \rf{bin} has in particular a solution for the exponent $p$, the
Theorem \ref{tbin} implies that $B < p^p$; when $c > 2$, combining this with the
previous lower bound implies that there are no solutions.
For $c = 2$, we deduce that $| X | < p$ and, after applying the
Theorem \ref{tbin} again and letting $\xi = \zeta^{1/p}$ be a
primitive \nth{p^2} root of unity, we obtain the equation
\[ Y^{p^2} = \frac{X^{p^2} - 1}{p^e(X^p-1)} = \Norm_{\Q[ \xi ]/\Q}
(\alpha) \quad \quad \quad \alpha = \frac{X-\xi}{(1-\xi)^e} .\] As
usual, the conjugates of the ideal $(\alpha)$ are pairwise coprime. We
let $\eu{A} = (Y, \alpha)$ be an ideal with $N(\eu{A}) = (Y)$; moreover,
if $\eu{L} | \eu{A}$ is a prime ideal and $N(\eu{L}) = (\ell)$, then
the rational prime $\ell$ is totally split in $\Q[ \xi ]$, the factors
being the primes $(\ell, \sigma_c(\alpha))$. Being totally split, it
follows in particular that $\ell \equiv 1 \bmod p^2$ so $Y \geq \ell >
2 p^2$, in contradiction with $Y < X < p+1$. This shows that there are
no solutions for $n = p^2$. \qed

\begin{corollary}
  If the Equation \rf{bin} in which $n = p^c$ is a prime power has non
  trivial solutions for which \rf{nosplit} holds, then $c = 1$.
\end{corollary} 

\qed

The primes dividing the exponent $n$ used in the above corollaries are
by definition coprime to $\varphi^*(B)$. As a consequence, if $n$ is
an exponent for which \rf{bin} has a solution and $m | n$ is the
largest factor of $n$ with $m \in \id{N}(B)$ -- as defined in
\rf{expos} -- then the corollaries imply that there is at most one
prime dividing $n/m$ and the exponent of this prime in the prime
decomposition of $n$ must be one. This is the first statement of
Theorem \ref{genexpo}, which thus follows from these corollaries and
Theorem \rf{bin}.

\section{Appendix}
The proof of Theorem \ref{tbin} is based on results from \cite{Mi1}. It has
been pointed out that the proof of Theorem 3 in \cite{Mi1} may require
some more detailed explanation in the case of a singular system of
equations in the proof of Lemma 14 of \cite{Mi1}. Since the statements
of \cite{Mi1} are correct and can even be slightly improved, while the
explanations may have seemed insufficient, we provide here for the readers
interested to understand the technicalities of the proofs in \cite{Mi1}
some additional details and explanation, confirming those claims and results.

\subsection{Clarification on the singular case of Theorem 3 in \cite{Mi1}}
Let $m \in \Z_{>0}$ be a positive integer, $\K$ a field, $V = \K^{m}$
as a $\K$-vector space and let $L \subsetneq V$ be a proper subspace of $V$ of dimension $r$. We
assume that there exists at least one vector $w_1 \in L$ which is free
of $0$-coefficients over the canonical base $\id{E}$. For $(x, y) \in V^2$, $x=(x_1, \cdots, x_m)$ and $y=(y_1,\cdots,y_m)$, the
Hadamard product is defined by $[ x, y ] = (x_1y_1, \cdots, x_m y_m)$. For
any subspace $W \subset V$ we define the $W$-\textit{bouquet} of $L$ by
\[ L_W = \langle \ \ \{ \ [ w, x ] \ : \ w \in W, \ x \in L \ \} \ \
\rangle_{\K} ,\] the $\K$-span of all the Hadamard products of
elements in $W$ by vectors from $L$.


\begin{lemma}
\label{bouquet}
Let $a_1 = (1, 1, \ldots, 1)$ over $\id{E}$, and $e_2 \in V$ such that its coordinates be pairwise distinct over $\id{E}$. Let $A_2 = \langle \{ a_1, a_2\} \rangle_{\K}$ be
 the subspace generated by $e_1,e_2$. Let $L_{A_2}$ be the
resulting $A_2$-bouquet. Then $dim(L_{A_2}) > dim(L)$.
\end{lemma}

\begin{proof}
  Obviously, $L \subset L_{A_2}$ (as $a_1 \in A$). We would like to
  show that $L_{A_2} \neq L$. We know that the system $(w_1,
  [w_1,a_2], [w_1,a_2^2], \cdots, [w_1,a_2^{m-1}])$ (the notion of
  power of a vector here is to be understood as an "Hadamard power")
  is free (as it induces a Vandermonde matrix over $\id{E}$, 
  $w_1$ does not have any zero among its coordinates and all
  coordinates of $a_2$ are pairwise distinct). We know that $w_1 \in
  L$; let us assume that $[w_1,a_2^i] \in L$ for $i \leq j$ (we know
  that $j \leq  m - r < m$). Then, $[w_1,a_2^j] \in L$ and
  $[w_1,a_2^{j+1}] \notin L$. However, the Hadamard product of
  $[w_1,a_2^j] \in L$ by $a_2$, that is $[w_1,a_2^{j+1}]$, belongs to
  $L_{A_2}$. Thus, $\dim L_{A_2} > \dim L$.
\end{proof}

\subsubsection{Application of Lemma \ref{bouquet} to the proof
    of the singular case in the argument on pages 266 -- 270 of \cite{Mi1}}

We apply here the lemma in the first case (that is $x \not\equiv s
\mod p$, where $s \in \{ -1, 0, 1\}$), the application to the second
case being similar.

Let all notation be like in Lemma 14 in \cite{Mi1}. As in \cite{Mi1},
we will assume that $\rg{A} = \left( \zeta^{-\kappa_{ac}/a}
\right)^{(p-1)/2}_{a,c=1}$ (where $\kappa_{ac}$ are the \textit{Galois
  exponents}) is singular. Let $m = (p-1)/2$, $\K = \Q(\zeta_p)$ and
$r = \rk (\rg{A}) < (p-1)/2$. Without loss of generality, we assume
that a regular $r$-submatrix of $\rg{A}$ is built with the first $r$
rows and the first $r$ columns. Therefore, the first $r$ rows of
$\rg{A}$ are independent, and we denote by $W$ the sub-space of
$V=\K^{m}$ generated by the first $r$ row vectors $w_1, \cdots, w_r$
of $\rg{A}$. For $a_1 = (1, 1, \ldots, 1)$, we let $a_2$ be the vector
of $V$ whose components are \footnote{\text{In the context of \cite{Mi1}, $\eta$ corresponds to $b_1[\theta]$ in our context}} $\left( \eta (\sigma_c
  \theta)\right)_{c=1}^{(p-1)/2}$ and $A_2=\{ a_1, a_2 \}$. Then,
according to Lemma \ref{bouquet}, there exists at least one vector
$\vec{v} \in L_{A_2}$ which is independent on the first $r$ vectors of
$\rg{A}$.

Let $\rg{S}$ be the $(r+1) \times (r+1)$ submatrix of $\rg{A}$
comprising the first $r$ rows and $r+1$ columns of $\rg{A}$, to which
we have added an additional row: the first $r+1$ components of
$\vec{v}$. Let $\vec{\lambda}'$ be the vector solution of $\rg{A}
\vec{\lambda}' = \vec{d}'$, where $\vec{d}' = \left( \delta_{c,r+1}
\right)_{c=1}^{r+1}$. We know that $\vec{\lambda}' \neq \vec{0}$, as
$\rg{S}$ is regular and $\vec{d}'$ is not the null vector. For $1 \leq
c \leq r + 1$, by Cramer's rule, $\lambda_c = \frac{S_c}{S}$, where
$S_c$ are the determinants of some minors of $\rg{S}$ obtained by
replacing the $c$-column by $\vec{d}'$, and $S = \det \rg{S}$.

Let $\vec{\lambda} \in V$ be a vector whose first $r+1$ coordinates
are those of $\vec{\lambda}'$ and the others are $0$. Let $\left(
  \delta_{c,r+1} \right)_{c=1}^{m}$. Then, $\vec{\lambda}$ verifies:
$\rg{A} \vec{\lambda} = \vec{d}$.

Let $\delta = \sum_{c=1}^{r+1} \left( \lambda_c \cdot \beta_c +
  \overline{\lambda_c \cdot \beta_c} \right)$. Using Hadamard's
inequality, we bound $|S_c| \leq \left(
  \frac{p-3}{2}\right)^{\frac{p-3}{4}} = D_1$ and $|S| \leq \left(
  \frac{p-1}{2}\right)^{\frac{p-1}{4}} = D_0$. Then, using the fact
that the choice of $\lambda_c$ eliminates the first term in the
expansion of $f_c$, we find that $|S| \cdot |\delta| \leq 2
x^{(p-1)/2p} \cdot \sum_{c=1}^{r+1} |S_c| |R_{c,0} (x)|$, where
$R_{c,0} (x) = f_c(x) - x^{(p-1)/2p}$. With the same arguments
as in \cite{Mi1}, we deduce: $$|S \delta| < 2(p-1) D_1 \cdot
\frac{1}{|x|^{(p+1)/2p}}.$$ This inequality holds for all conjugates
$\sigma_c(\delta)$, thus leading to:

\begin{equation*}
  \left|\Norm\left(S \delta\right)\right| <  \left( 2 (p-1) D_1\right)^{(p-1)/2} 
\cdot \frac{1}{|x|^{\frac{(p-1)(p+1)}{4p}}}.
\end{equation*}

If $\delta \neq 0$, then $\left|\Norm\left(S \delta\right)\right| \geq
1$ and thus $|x| \leq 2^{5-p} \left( \frac{p}{2}
\right)^{\frac{p}{2}}$. If $\delta = 0$, then $0 = S \delta = S \cdot
|x|^{(p-1)/2} - \sum_{c=1}^{(p-1)/2} S_c R_{0,c}$, and thus:

\begin{equation*}
  |x| \leq \sum_c |S_c|/|S| < (p-1) D_1 < 3
  \left(\frac{p-3}{2}\right)^{(p-3)/2} .
\end{equation*}
These bounds are better than the ones in \cite{Mi1}, and this
concludes the clarification.

\subsection{Proof of Lemma \ref{simple}}
\begin{proof}
  Let $\theta = \sigma_w \psi_u + \sigma_z \psi_v$. The conditions
  required by the lemma lead to the following linear system of
  equations over $\F_n$:
\begin{equation}
\label{sys1}
\left\{
\begin{array}{lcllcl}
  w \cdot \varphi (\psi_u) & + &  z \cdot \varphi (\psi_v) & = & 0 \\
  1/w \cdot \phi^{(-1)} (\psi_u) & + & 1/z \cdot \phi^{(-1)} (\psi_v) & \neq & 0
\end{array}
\right.
\end{equation}
Considered as a linear system in the unknowns $w, z \in \F_n$, the
above system has the matrix $M = \left( \begin{array}{c c} \varphi
    (\psi_u) & \varphi (\psi_v)\\ \phi^{(-1)} (\psi_v) & \phi^{(-1)}
    (\psi_u) \end{array}\right)$. Assume that the product $P(t) =
\varphi(\psi_t) \cdot \phi^{(-1)}(\psi_t)$ is not constant for all $t
\in \ZMs{n}$. Then there are two elements $u, v \in \ZMs{n}$ such that
$P(u) \neq P(v)$; for such values $u, v$, the matrix $M$ is regular
over $\F_p$ and for any non vanishing right hand side in the second
equation, the system has a unique solution $(w, z)$. For this choice
of $u, v; w, z$, the element $\theta = \sigma_w \psi_u + \sigma_z
\psi_v$ satisfies the condition of the lemma.

We now show that $P(t) : \ZMs{n} \ra \F_p$ is not a constant
function. The proof uses explicit computations which include divisions
by several constants which must be assumed to be non - null. Therefore
we suppose that $n \not \in E := \{ 3, 7 \}$ and shall verify
independently that the claim of the lemma holds for this exceptional
set.

Let $\varphi$ be the Fermat quotient map and $\Theta_k$ be the \nth{k}
Fuchsian.  For any integer $1 < k < n-1$, we have:

\begin{equation*}
\begin{array}{lcll}
(n-k)^n - (n-k) & \equiv & -k^n - n + k & \mod n^2 \\
 & \equiv & -n \left( \frac{k^n - k}{n} + 1 \right) & \mod n^2 .
\end{array}
\end{equation*}
Dividing both terms by $n$ and recalling from Lemma \ref{Jha} that 
$\varphi(\Theta_k) = \varphi(k) \equiv \frac{k^n - k}{n}  \bmod n$, we find:
\begin{equation}
\label{FuchsEq}
\varphi (\Theta_{n-k}) = n - \left( 1 + \varphi(\Theta_k) \right).
\end{equation}

Using now \rf{voronoi} from Lemma \ref{Jha}, with $m=2$, we find that:
\begin{equation*}
\phi^{(-1)} (\Theta_k) \equiv \frac{k^3 - k}{2k^2} B_{2}  \equiv \frac{1}{12}\cdot 
\left(k - \frac{1}{k} \right)  \mod n,
\end{equation*}
where we used the fact that $B_2 = 1/6$. Finally, using that $\psi_k =
\Theta_{k+1} - \Theta_k$ for $k > 1$ while $\psi_1 = \Theta_2$, we
obtain the following expressions for the moments of interest:
\begin{eqnarray*}
  \varphi(\psi_k) & = & \varphi(k+1)-\varphi(k), \\
  \phi^{(-1)}(\psi_k) & \equiv & \frac{1}{12} \cdot \left( 1 + \frac{1}{k(k+1)}\right).
\end{eqnarray*}
Note that $\phi^{(-1)}(\psi_k) = 0$ iff $k^2 + k + 1 = 0$; if $n
\equiv 1 \bmod 6$, the equation has two solutions in $\F_n$, otherwise
it has none. In the latter case $\phi^{(-1)}(\psi_k) \neq 0$ for all
k.

We shall assume that $P$ is the constant function and shall show that
this assumption fully determines the Fermat quotient of integers in
dependence of $\varphi(2)$, and this determination is in contradiction
with \rf{FuchsEq}; the contradiction implies that $P$ cannot be
constant, thus completing the proof.

Let thus $C = \varphi(2) \cdot \phi^{(-1)}(\Theta_2) = \varphi(2)
\cdot \frac{1}{8}$. Assume first that $\varphi(2) = 0$ and recall from
\rf{FuchsEq} that $\varphi(k) + \varphi(n-k) + 1 = 0$.  Therefore at
least $\frac{n-1}{2}$ of the values of $\varphi$ are
non-vanishing. Since $\phi^{(-1)}(k) \cdot (\varphi(k+1)-\varphi(k)) =
0$ for all $k$ we see that if $n \not \equiv 1 \bmod 6$, then
$\varphi$ is constantly vanishing, which is impossible.

If $n \equiv 1 \bmod 6$, let $l, m \in \F_n$ be the non trivial third
roots of unity, so $\phi^{(-1)}(\psi_l) = \phi^{(-1)}(\psi_m) = 0$,
while for all $k \not \in \{l, m\}$ we must have $\varphi(k+1) =
\varphi(k)$. In particular, if $l < m$, there are two integers $a, b$
such that
\[ \varphi(2) = 0 = \ldots = \varphi(l); \quad \varphi(l+1) = a =
\ldots = \varphi(m); \quad \varphi(m+1) = b = \ldots = \varphi(n-1).\]
But $\varphi(n-1) = -1$ while $\varphi(n-2) = -1 - \varphi(2) = -1$,
so $b = -1$. For symmetry reasons induced by \rf{FuchsEq}, we must
have $a = -1/2$ and $m = n-l$. This is absurd since $m^3 \equiv 1
\bmod n$ implies $l^3 = (n-m)^3 \equiv -m^3 \equiv -1 \bmod n$, so $n
= 2 \not \equiv 1 \bmod 6$. Thus $\varphi(2) \neq 0$ in this case
too. Since $\phi^{(-1)}(l) = 0$, it follows however that $C =
\varphi(l) \cdot \phi^{(-1)}(l) = 0$ and thus $C = 0 = \varphi(2) / 8$
and we should have $\varphi(2) = 0$, in contradiction with the facts
established above.  Consequently, if $n \equiv 1 \bmod 6$, then $P$
cannot be constant.

We consider now the case $n \not \equiv 1 \bmod 6$, in which we know
that $C \neq 0$. By expressing $C = P(2) = P(k)$ we obtain the
following induction formula
\begin{eqnarray*}
  C = \frac{1}{12} \cdot \frac{3 \varphi(2)}{2} & = & 
\frac{1}{12} (\varphi(k+1) - \varphi(k)) \cdot \frac{k^2 + k + 1}{k(k+1)}, \quad \hbox{hence} \\
  \varphi(k+1) - \varphi(k) & = & \frac{3 \varphi(2)}{2} \cdot \frac{k(k+1)}{k^2 + k + 1}, \\
  \varphi(3) - \varphi(2) & = & \frac{9}{7} \varphi(2) \quad \Rightarrow \quad \varphi(3) = \frac{16}{7} \varphi(2).
\end{eqnarray*}
By eliminating $\varphi(2)$ from the above identity for two successive
values of $k$ one finds
\[ \varphi(k+1) = \frac{2k^3}{k^3-1} \cdot \varphi(k) +
\frac{k^3+1}{k^3-1} \cdot \varphi(k-1). \]

We shall use the reflexion formula \rf{FuchsEq} between the last and
the first values in the sequence $1, 2,\ldots, n-2, n-1$. Letting $k =
n-2$ in the above induction, we find
\begin{eqnarray*}
  -1 & \equiv & \varphi(n-1) \equiv \frac{16}{9} \cdot \varphi(n-2) + 
\frac{7}{9} \cdot \varphi(n-3) \\
& \equiv & \frac{16}{9} \cdot(-1- \varphi(2)) + \frac{7}{9} \cdot (-1-\varphi(3)) \bmod n, \\
  9 & \equiv & 16 + 16 \varphi(2) + 7 + 7 \varphi(3) \equiv 23 + 
(16+7 \cdot \frac{16}{7} ) \varphi(2) \bmod n, \quad \hbox{hence} \\
  -7 & \equiv &  16  \cdot \varphi(2)  \bmod n.
\end{eqnarray*}
Consequently $\varphi(2) \equiv -\frac{7}{16} \bmod n$ and thus
$\varphi(3) \equiv \frac{16}{7} \varphi(2) \equiv -1 \bmod n$. But
then \rf{FuchsEq} implies that $\varphi(n-3) = -1-\varphi(3) = 0$, and
thus $C = 0$, in contradiction with the previously obtained non
vanishing fact. This confirms that $P(t)$ is non constant in
this case too.

It remains to verify the claim for the exceptional primes in $E$. For
$n = 3$ the Stickelberger ideal is trivial, so there is nothing to
prove. For $n = 7$ one can repeat the proof of the case $n \equiv 1
\bmod 6$, which requires no division by $7$; this completes the proof
of the Lemma.
\end{proof}

\thanks{\textbf{Acknowledgements}: \textit{The first author is grateful to the Universities of Bordeaux and
  G\"ottingen for providing a stimulating environment during the
  development of this work. Both authors thank Mike Bennett and
  Kalman Gy\H{o}ry for suggesting this interesting problem for an
  algebraic investigation.}}


\begin{thebibliography}{ZZZZZZ}

\bibitem[Be]{Be} M. A. Bennet: {\em Rational Approximation To
    Algebraic Numbers Of Small Height: The Diophantine Equation
    $|ax^n-by^n|=1$}, J. Reine Angew. Math., \textbf{535},
  pp. 1--49 (2001).

\bibitem[BGMP]{BGMP} M. A. Bennett, K. Gy\H{o}ry, M. Mignotte and
  \'A. Pint\'er: {\em Binomial Thue equations and polynomial powers},
  Compositio Math. \textbf{142}, pp. 1103--1121 (2006).

\bibitem[BBM]{Bilu} Y. Bilu, Y. Bugeaud and M. Mignotte: {\em The
    problem of Catalan}, Springer (2014).

\bibitem[Bilu]{Bilu2} Y. Bilu: {\em Catalan's conjecture (after
    Mih\u{a}ilescu)}, S\'eminaire Bourbaki, Expos\'e 909, 55\'eme
  ann\'ee (2002-2003); Ast\'erisque 294, pp. 1--26 (2004).

\bibitem[BHM]{BHM} Y. Bugeaud, G. Hanrot and M. Mignotte: {\em Sur
    l'\'{e}quation diophantienne $\frac{x^{n}-1}{x-1} = y^{q}$}, III,
  Proc. London. Math. Soc. \textbf{84}, pp. 59--78 (2002).

\bibitem[BH]{BH} J. P. Buhler and D. Harvey: {\em Irregular primes to
    163 million}, Math. Comp., \textbf{80:276}, pp. 2435--2444 (2011).

\bibitem[G]{G} K. Gy\H{o}ry: {\em Personal communication.}

\bibitem[Jha]{Jha} V. Jha: {\em The Stickelberger Ideal in the Spirit
    of Kummer with Application to the First Case of Fermat's Last
    Theorem}, Queen's University, Queen's papers in pure and applied
  mathematics \textbf{93} (1993).

\bibitem[Mi1]{Mi2} P. Mih\u{a}ilescu: {\em Primary cyclotomic units
    and a proof of Catalan's conjecture}, in \textit{J. Reine
    Angew. Math.} \textbf{572}, pp. 167--195 (2004).

\bibitem[Mi2]{Mi1} P. Mih\u{a}ilescu: {\em Class Number Conditions for
    the Diagonal Case of the Equation of Nagell and Ljunggren}, in
  \textit{Diophantine Approximation}, Springer Verlag, Development in
  Mathematics \textbf{16}, pp. 245--273 (2008).

\end{thebibliography}
\end{document}